\documentclass[11pt]{amsart}

\usepackage{epigamath-voisin}

\usepackage[english]{babel}

\numberwithin{equation}{section}

\usepackage[shortlabels]{enumitem}
\setlist[enumerate,1]{label={\rm(\arabic*)}, ref={\rm\arabic*}} 

\usepackage{amsmath, amsthm, amssymb, amsfonts}
\usepackage[normalem]{ulem}
\usepackage{hyperref}
\usepackage{mathrsfs}
\usepackage{verbatim} 
\usepackage{longtable}

\usepackage{mathtools}

\usepackage{tikz}
\usetikzlibrary{decorations.pathmorphing}
\tikzset{snake it/.style={decorate, decoration=snake}}

\usepackage{caption}

\usepackage{tikz-cd}
\usetikzlibrary{arrows}

\theoremstyle{plain}
\newtheorem{thm}{Theorem}[section]
\newtheorem{cor}[thm]{Corollary}
\newtheorem{lem}[thm]{Lemma}
\newtheorem{prop}[thm]{Proposition}
\newtheorem{conj}[thm]{Conjecture}

\theoremstyle{definition}

\theoremstyle{remark}
\newtheorem{rmk}[thm]{Remark}

\newcommand{\BA}{{\mathbb{A}}}

\newcommand{\BC}{{\mathbb{C}}}

\newcommand{\BP}{{\mathbb{P}}}

\newcommand{\BQ}{{\mathbb{Q}}}

\newcommand{\BZ}{{\mathbb{Z}}}

\newcommand{\CD}{{\mathcal D}}

\newcommand{\CF}{{\mathcal F}}
\newcommand{\CG}{{\mathcal G}}
\newcommand{\CH}{{\mathcal H}}

\newcommand{\CO}{{\mathcal O}}
\newcommand{\CP}{{\mathcal P}}
\newcommand{\CQ}{{\mathcal Q}}

\newcommand{\CV}{{\mathcal V}}
\newcommand{\CW}{{\mathcal W}}

\DeclareFontFamily{OT1}{rsfs}{}
\DeclareFontShape{OT1}{rsfs}{n}{it}{<-> rsfs10}{}
\DeclareMathAlphabet{\curly}{OT1}{rsfs}{n}{it}

\def\lra{\longrightarrow}

\DeclareMathOperator{\Coh}{Coh}

\DeclareMathOperator{\Perv}{Perv}
\DeclareMathOperator{\DR}{DR}
\DeclareMathOperator{\gr}{gr}
\DeclareMathOperator{\Sym}{Sym}
\DeclareMathOperator{\HM}{HM}

\DeclareMathOperator{\codim}{codim}
\DeclareMathOperator{\id}{id}
\DeclareMathOperator{\pr}{pr}
\DeclareMathOperator{\even}{even}

\newcommand{\supth}[1]{\ensuremath{#1^{\mathrm{th}}}}

\YearArticle{2023} \EpigaArticleNr{6} \ReceivedOn{May 27, 2022}
\InFinalFormOn{June 2, 2023}
\AcceptedOn{June 18, 2023}

\title{Perverse--Hodge complexes for Lagrangian fibrations}
\titlemark{Perverse--Hodge complexes for Lagrangian fibrations}

\author{Junliang Shen}
\address{Department of Mathematics, Yale University, New Haven, CT 06511, USA}
\email{junliang.shen@yale.edu}
\author{Qizheng Yin}
\address{BICMR, Peking University, Beijing 100871, China}
\email{qizheng@math.pku.edu.cn}

\authormark{J.~Shen and Q.~Yin}

\AbstractInEnglish{Perverse--Hodge complexes are objects in the derived category of coherent sheaves obtained from Hodge modules associated with Saito's decomposition theorem. We study perverse--Hodge complexes for Lagrangian fibrations and propose a symmetry between them. This conjectural symmetry categorifies the authors' ``perverse = Hodge'' identity and specializes to Matsushita's theorem on the higher direct images of the structure sheaf. We verify our conjecture in several cases by making connections with variations of Hodge structures, Hilbert schemes, and Looijenga--Lunts--Verbitsky Lie algebras.}



\acknowledgement{J.\,S.\ is supported by the NSF grant DMS-2134315. Q.\,Y.\ is supported by the NSFC grants 11831013 and 11890661.}

\dedication{To Claire Voisin, with admiration}

\begin{document}

%%%%%%%%%%%%%%%%%%%%%%%%%%%%%%%
% Title page
%%%%%%%%%%%%%%%%%%%%%%%%%%%%%%%

\maketitle

\begin{prelims}

\DisplayAbstractInEnglish

\bigskip

\DisplayKeyWords

\medskip

\DisplayMSCclass

\end{prelims}

%%%%%%%%%%%%%%%%%%%%%
% Table of Contents
%%%%%%%%%%%%%%%%%%%%%

\newpage

\setcounter{tocdepth}{1}

\tableofcontents

%%%%%%%%%%%%%%%%%%%%%
% Content begins here
%%%%%%%%%%%%%%%%%%%%%

\section{Introduction}

\subsection{Perverse--Hodge symmetry}
For a compact irreducible symplectic variety\footnote{We say that $M$ is irreducible symplectic if it is a compact K\"ahler manifold such that $H^0(M,\Omega_M^2)$ is spanned by a nowhere degenerate $2$-form.} $M$ of dimension $2n$ with a Lagrangian fibration $\pi\colon M\to B$, the decomposition theorem, \textit{cf.}~\cite{BBD}, 
\begin{equation}\label{decomp0}
R\pi_* \BQ_M[2n]  \simeq \bigoplus_{i=-n}^n P_i[-i], \quad P_i = {^\mathfrak{p}\CH}^i\left( R\pi_* \BQ_M[2n] \right)\in \Perv(B)
\end{equation}
provides important invariants for the topology of $\pi$. A \emph{perverse--Hodge}
symmetry was proven in \cite{SY}, connecting the cohomology of the perverse sheaves $P_i$ with the Hodge numbers of $M$.

\begin{thm}[\textit{cf.} \cite{SY}]\label{thm1}
We have
\begin{equation}\label{P=F}
    h^{j-n}(B, P_{i-n}) = h^{i,j}(M).
\end{equation}
Here $h^*(-)$ stands for $\dim H^*(-)$, and $h^{i,j}(-)$ denotes the Hodge number.
\end{thm}

The identity \eqref{P=F} governs the cohomology of the Lagrangian base, the invariant cohomology of a nonsingular fiber of $\pi$, and the Gokapumar--Vafa theory of $K3$ surfaces; we refer to \cite{SY, FSY, HLSY, HM} for more discussions on Theorem~\ref{thm1} and its applications.

The purpose of this paper is to explore and propose a \emph{categorification} of the perverse--Hodge symmetry. It suggests that Theorem~\ref{thm1} should conceptually be viewed as a cohomological shadow of a sheaf-theoretic symmetry for Lagrangian fibrations with possibly noncompact ambient spaces $M$. It is a mysterious phenomenon since all existing proofs of~\eqref{P=F}, \textit{cf.}~\cite{SY, HLSY, HM}, rely heavily on the global cohomological properties of compact irreducible symplectic manifolds, and they do not ``explain'' why such a categorification should exist. Our formulation uses perverse--Hodge complexes constructed from Hodge modules. 

\subsection{Perverse--Hodge complexes}

Let $(M,\sigma)$ be a nonsingular quasi-projective symplectic variety of dimension $2n$. Here $\sigma$ is a \emph{closed} nowhere degenerate holomorphic~$2$-form on $M$. Let~$\pi\colon M \to B$ be a proper \emph{Lagrangian fibration} onto a nonsingular base~$B$; \textit{i.e.}, the restriction of the symplectic form $\sigma$ to regular part of a fiber vanishes. Interesting examples of $\pi$ include Lagrangian fibrations of compact irreducible symplectic varieties, \textit{cf.}~\cite{Be}, and Hitchin's integrable systems, \textit{cf.}~\cite{Hit, Hit1}.

By Saito's theory \cite{S1, S2}, the decomposition theorem \eqref{decomp0} can be upgraded to an identity in the bounded derived category of Hodge modules. Let $\BQ_M^H[2n]$ be the trivial Hodge module, \textit{i.e.}, the pure Hodge module associated with the shifted trivial local system~$\BQ_M[2n]$. We have
\begin{equation}\label{DT}
\pi_+\BQ_M^H[2n] \simeq \bigoplus_{i=-n}^{n} P_i^H[-i],\quad P_i^H = \CH^i\left(\pi_+\BQ_M^H[2n]\right).
\end{equation}
The Hodge module $P_i^H$ consists of a regular holonomic (left-)$\CD_B$-module $\CP_i$ equipped with a good filtration~$F_\bullet \CP_i$; it corresponds to the perverse sheaf $P_i$ under the Riemann--Hilbert correspondence. The increasing filtration $F_\bullet \CP_i$ induces an increasing filtration on the de Rham complex
\[
\DR(\CP_i) = \left[\CP_i \lra \CP_i\otimes\Omega^1_B \lra  \cdots \lra \CP_i \otimes \Omega^n_B\right]\![n].
\]
The associated graded pieces are natural objects in the bounded derived category of coherent sheaves on $B$,
\[
\gr^F_k \DR(\CP_i) \in D^b\Coh(B).
\]

Up to re-indexing and shifting, we define
\[
\CG_{i, k} := \gr^F_{-k} \DR(\CP_{i - n})[n - i].\]
We call $\CG_{i,k}$ the \emph{perverse--Hodge complexes} associated with the Lagrangian fibration $\pi$; here~$i$ is the perverse degree, and $k$ is the Hodge degree. The object $\CG_{i,k}$ is nontrivial only if~\mbox{$0 \leq i,k \leq 2n$}.

Our main proposal is the following conjectural symmetry between perverse--Hodge complexes.

\begin{conj}\label{conj}
Let $\pi\colon M\to B$ be a Lagrangian fibration. We have
\[
\CG_{i, k} \simeq \CG_{k, i} \in D^b\Coh(B).
\]
\end{conj}

As in Theorem~\ref{thm0.6}, Conjecture~\ref{conj} categorifies a refined version of Theorem~\ref{thm1} when $M$ is a compact irreducible symplectic variety. By Proposition~\ref{prop2.3}, it also recovers Matsushita's result \cite{Ma3} on the higher direct images of $\CO_M$.

\subsection{Main results}
We provide evidence for Conjecture~\ref{conj} and verify it in several cases.

\subsubsection{Smooth morphisms}
Our first theorem verifies Conjecture~\ref{conj} when $\pi\colon M \to B$ is smooth. In fact, we obtain a stronger result in this case.

\begin{thm}\label{thm0.3}
Assume that $\pi\colon M \to B$ is smooth. The symplectic form $\sigma$ on~$M$ together with a po\-lar\-i\-za\-tion induces an isomorphism
\[
\CG_{i, k} \xrightarrow{\;\simeq\;} \CG_{k,i}
\]
at the level of complexes.
\end{thm}
Theorem~\ref{thm0.3} is essentially a reformulation of a result of Donagi and Markman \cite{DM} on the polarized variation of Hodge structures associated with the family. Both complexes~$\CG_{i,k}$ and~$\CG_{k,i}$ have the same length, and a term-by-term isomorphism is constructed between them. The mystery of Conjecture~\ref{conj} is an ``extension'' of this isomorphism to the singular fibers. As we see from Section~\ref{Sect2.4}, in general such an extension is complicated, and the derived category is essential for the formulation.

\subsubsection{Hilbert schemes}
Next, we consider the Lagrangian fibration
\begin{equation}\label{Hilb}
\pi^{[n]}\colon S^{[n]} \lra S^{(n)} \lra B^{(n)}
\end{equation}
induced by an elliptic fibration of a symplectic surface $\pi\colon S \to B$. Typical examples include: 
\begin{enumerate}[(i)]
    \item $\pi\colon S \to \BP^1$ is an elliptic $K3$, and $\pi^{[n]}\colon S^{[n]} \to (\BP^1)^{(n)} = \BP^n$ is a Lagrangian fibration of the compact irreducible symplectic variety $S^{[n]}$; and
    \item $\pi\colon  S \to \BA^1$ and the induced morphisms $\pi^{[n]}\colon  S^{[n]} \to (\BA^{1})^{(n)} = \BA^n$ are the Hitchin fibrations associated with five families of moduli of parabolic Higgs bundles labeled by certain affine Dykin diagrams; \textit{cf.}~\cite{Gr,Z}.
\end{enumerate}

\begin{thm}\label{thm0.4}
Conjecture~\ref{conj} holds for \eqref{Hilb} for any $n \geq 1$.
\end{thm}

The decomposition theorem associated with \eqref{Hilb} has many supports besides the full base~$B^{(n)}$. In particular, the isomorphism of Conjecture~\ref{conj} in this case is not merely an extension of the isomorphism of Theorem~\ref{thm0.3} for variations of Hodge structures. Semisimple objects of the decomposition theorem \eqref{DT} supported on the ``boundary'' of $B^{(n)}$ contribute nontrivially.

\subsubsection{Global cohomology}
Lastly, we consider Lagrangian fibrations $\pi\colon  M \to B$ associated with compact irreducible symplectic varieties.\footnote{As the base $B$ is assumed to be nonsingular, by a result of Hwang \cite{Hw} we know that $B$ is a projective space. However, this fact will not be used in this paper.} Since $B$ is projective, the (hyper-)cohomology groups of the perverse--Hodge complexes are finite-dimensional. 

The following theorem shows that in this case Conjecture~\ref{conj} holds cohomologically.

\begin{thm}\label{thm0.6}
Let $\pi\colon  M \to B$ be a Lagrangian fibration with $M$ a compact irreducible symplectic variety. Then we have
\begin{equation}\label{P=F2}
H^*(B, \CG_{i,k}) \simeq H^*(B, \CG_{k,i}).
\end{equation}
\end{thm}

We prove Theorem~\ref{thm0.6} following the ideas of \cite{SY}, which connects the cohomology groups in \eqref{P=F2} to the weight spaces of the Looijenga--Lunts--Verbitsky algebra; \textit{cf.}~\cite{LL,Ver90,Ver95,Ver96}. As a byproduct we deduce that \eqref{P=F2} refines \eqref{P=F}, which justifies that Conjecture~\ref{conj} categorifies Theorem~\ref{thm1}. 

From another aspect, Theorem~\ref{thm0.6} suggests that, among all the symmetries encoded by the Looijenga--Lunts--Verbitsky algebra of $M$, the particular one inducing \eqref{P=F2} can be lifted sheaf-theoretically.

\subsection*{Acknowledgments}
We are grateful to Davesh Maulik for his enthusiasm and for many helpful discussions. We also thank Bohan Fang, Mirko Mauri, Peng Shan, and the anonymous referee for useful comments and suggestions.

\section{Smooth morphisms and variations of Hodge structures}\label{Sec1}

Throughout this section, we assume that $\pi\colon  M \to B$ is smooth, so that the Hodge modules~$P_i^H$ are variations of Hodge structures. 

\subsection{Variations of Hodge structures}
As a consequence of the Arnold--Liouville theorem, a nonsingular fiber of a Lagrangian fibration is a complex torus. In particular, the smooth map $\pi\colon  M \to B$ is a family of abelian varieties. The key to understanding the topology of $\pi$ is the variation of Hodge structures
\[
V = R^1\pi_*\BQ_M;
\]
it is polarized of weight $1$ with associated holomorphic vector bundle $\CV = V \otimes_\BQ \CO_B$. The integrable connection $\nabla\colon  \CV \to \CV\otimes \Omega_B^1$ and the Hodge filtration
\begin{equation}\label{Hodge_fil}
 0 = F^2\CV \subset F^1\CV \subset F^0\CV = \CV
\end{equation}
are compatible via the Griffiths transversality relation
\[
\nabla(F^i\CV) \subset F^{i-1}\CV \otimes \Omega_B^1.
\]
This yields an $\CO_B$-linear map between the graded pieces of \eqref{Hodge_fil}
\begin{equation}\label{nabla_bar}
\overline{\nabla}\colon  \CV^{1,0} \lra \CV^{\,0,1}\otimes \Omega_B^1.
\end{equation}
Here $\CV^{\,i,1-i} = \gr_F^i\CV$ is a vector bundle describing the variation of $H^{i,1-i}({M_b})$ of the fibers $M_b$ with $b \in B$. 

For our purpose, we also consider the variation of Hodge structures $\CV^k = \wedge^k V$ of weight $k$. Its Hodge filtration is 
\[
0 = F^{k+1} \CV^{\mkern1mu k} \subset F^{k-1} \CV^{\mkern1mu k} \subset \cdots \subset F^0\CV^{\mkern1mu k} = \CV^{\mkern1mu k}, 
\]
where the $\supth{i}$ piece is given by
\[
F^i\CV^{\mkern1mu k} = \sum_{{i_l+i_2+\dots +i_k}= i} F_{i_1}\CV \wedge F_{i_2} \CV \wedge \cdots \wedge F_{i_k}\CV.
\]
We denote by $\CV^{\,i,j}$ the graded piece $\gr_i^F \CV^{\,i+j}$.

\begin{lem}\label{lem1.1}
We have a canonical isomorphism of vector bundles
\[
\wedge^i\CV^{1,0} \otimes \wedge^j \CV^{\,0,1} \xrightarrow{\;\simeq\;} \CV^{\,i,j}.
\]
\end{lem}

\begin{proof}
  The morphism is induced by the cup product and the compatibility with Hodge filtrations. It suffices to check that it is an isomorphism when
  restricting to each $b \in B$; this follows from the fact that $M_b$ is an abelian variety, so that we have
\[
H^{i,j}(M_b) = \wedge^i H^{1,0}(M_b) \otimes \wedge^j H^{0,1}(M_b). \qedhere
\]
\end{proof}

\subsection{Symplectic form}
We discuss the interplay between the symplectic form~$\sigma$ and the variation of Hodge structures $\CV$.% This plays a crucial role in proving Theorem~\ref{thm0.3}.

By \cite[Lemma 2.6]{Ma3}, the symplectic form $\sigma$ and a polarization on $M$ induce an isomorphism
\begin{equation}\label{sigma}
    \iota\colon  \CV^{\,0,1} \xrightarrow{\;\simeq\;} \Omega^1_B,
\end{equation}
which further yields
\begin{equation*}
  \wedge^k\iota\colon  \CV^{\,0,k} = \wedge^k \CV^{\,0,1} \xrightarrow{\;\simeq\;} \wedge^k \Omega_B^1  = \Omega_B^k. 
\end{equation*}

Combining \eqref{nabla_bar} and \eqref{sigma}, we obtain a morphism of vector bundles as the composition:
\[
\theta =  (\iota \otimes 1)\circ \overline{\nabla}  \colon  \CV^{1,0} \lra \CV^{\,0,1} \otimes \Omega_B^1 \lra \Omega_B^1 \otimes \Omega_B^1.
\]

\begin{lem}[Donagi--Markman]\label{lem1.2}
The morphism $\theta\colon  \CV^{1,0} \to \Omega_B^1 \otimes \Omega_B^1$ is symmetric with respect to the two factors of the target.
\end{lem}

\begin{proof}
Notice that $\CV^{1,0}$ is dual to $\CV^{\,0,1} \simeq \Omega_B^1$ via the polarization. Hence $\theta$ can be viewed as a section of $\Omega_B^1 \otimes \Omega_B^1 \otimes \Omega_B^1$. The proposition follows from the cubic condition for Lagrangian fibration \cite[Lemma~7.5]{DM} which says that the section corresponding to $\theta$ is induced by a section of $\Sym^3\Omega_B^1$; see also \cite[Theorem 4.4]{Voisin}.
\end{proof}

More generally, for any $k \geq 1$ we consider the morphism
\[
\overline{\nabla}\colon  \CV^{\mkern1mu k,0} \lra \CV^{\mkern1mu k-1,1}\otimes \Omega_B^1 = \CV^{\mkern1mu k-1,0}\otimes \CV^{\,0,1} \otimes \Omega_B^1,
\]
where the first map is induced by the Gauss--Manin connection $\nabla\colon  \CV \to \CV\otimes \Omega_B^1$ and the second identity is given by Lemma~\ref{lem1.1}.

\begin{cor}\label{cor1.3}
The composition 
\[
(1\otimes \iota \otimes 1)\circ \overline{\nabla}\colon   \CV^{\mkern1mu k,0} \lra  \CV^{\mkern1mu k-1,0}\otimes \CV^{\, 0,1} \otimes \Omega_B^1 \lra \CV^{\mkern1mu k-1,0}\otimes \Omega_B^1 \otimes \Omega_B^1 
\]
is symmetric with respect to the second and  third factors of the target.
\end{cor}

\begin{proof}
We proceed by induction on $k$. The induction base is Lemma~\ref{lem1.1}. Now assume that the statement holds for $k-1$. We have
\[
\CV^{\mkern1mu k,0} = \wedge^{k-1} \CV^{1,0} \wedge \CV^{1,0} = \CV^{\mkern1mu k-1,0} \wedge \CV^{1,0}.
\]
We consider a local section $s_{k}$ of $\CV^{\mkern1mu k,0}$ which can be written as $s_{k-1} \wedge t$ with $s_{k-1}$ and $t$ local sections of $\CV^{\mkern1mu k-1,0}$ and $\CV^{1,0}$, respectively. The image $\overline{\nabla}(s_{k})$ consists of two terms $\overline{\nabla}(s_{k-1})\wedge s_1$ and $s_{k-1} \wedge \overline{\nabla}(s_1)$. We obtain from the induction hypothesis and the induction base that both of them are local sections of $\CV^{\mkern1mu k-1,0}\otimes \Sym^2\Omega_B^1$. This completes the induction. 
\end{proof}

\subsection{Proof of Theorem~\ref{thm0.3}} 
The main ingredients of the proof are
\begin{enumerate}[(i)]
    \item\label{pf-1} the isomorphism \eqref{sigma} induced by the symplectic form $\sigma$ and a polarization, and
    \item\label{pf-2} the symmetry of Corollary~\ref{cor1.3}, which follows from the Donagi--Markman cubic condition.
\end{enumerate}

We first note that by Lemma~\ref{lem1.1} we have a canonical isomorphism
\begin{equation}\label{1.3_1}
\CV^{\,i,j} \otimes \Omega_B^k = \CV^{\,i,0} \otimes \wedge^j \CV^{\,0,1} \otimes \wedge^k \Omega_B^1.
\end{equation}
Hence \eqref{sigma} induces an isomorphism of vector bundles
\begin{equation}\label{ijk}
    \iota_{i,j,k} \colon  \CV^{\,i,j} \otimes \Omega_B^k \xrightarrow{\;\simeq\;} \CV^{\,i,k} \otimes \Omega_B^j
\end{equation}
by switching the second and third factors of the right-hand side of \eqref{1.3_1}.

Secondly, for any $i,j,k$, the Gauss--Manin connection of $\CV$ induces an $\CO_B$-linear morphism
\[
\overline{\nabla}\colon   \CV^{\,i,j} \otimes \Omega_B^k \lra \CV^{\,i-1,j+1}\otimes \Omega_B^{k+1}. 
\]

The following proposition shows the compatibility between the isomorphisms $\iota_{i,j,k}$ and the morphisms $\overline{\nabla}$; it relies heavily on  ingredient~\ref{pf-2}.

\begin{prop}\label{prop1.4}
We have a commutative diagram
\[
\begin{tikzcd}
 \CV^{\,i,j} \otimes \Omega_B^k \arrow[r,"\overline{\nabla}"] \arrow[d, "\iota_{i,j,k}"]
& \CV^{\,i-1,j+1}\otimes \Omega_B^{k+1} \arrow[d, "\iota_{i-1,j+1,k+1}"] \\
\CV^{\,i,k}\otimes \Omega_B^j \arrow[r, "\overline{\nabla}"]
& \CV^{\,i-1,k+1}\otimes \Omega_B^{j+1}\rlap{.}
\end{tikzcd}
\]
\end{prop}

\begin{proof}
To simplify the notation, we write the morphism of the top horizontal arrow as
\begin{equation}\label{4.1_0}
\overline{\nabla}\colon  \CV^{\,i,0} \otimes \Omega_B^j \otimes \Omega_B^k \lra \CV^{\,i-1,0} \otimes \Omega_B^{j+1} \otimes \Omega_B^{k+1}, 
\end{equation}
where we suppress the isomorphisms induced by \eqref{sigma} and Lemma~\ref{lem1.1}. In particular, for a local section
\begin{equation}\label{1.4_1}
s \otimes t \otimes u \in \Gamma\left( \CV^{\,i,0} \otimes \Omega_B^j \otimes \Omega_B^k \right),
\end{equation}
the image $\overline{\nabla}(s\otimes t \otimes u)$ is a local section of $\CV^{\,i-1,0} \otimes \Omega_B^{j+1} \otimes \Omega_B^{k+1}$. Similarly, for the same element \eqref{1.4_1} we have
\[
\overline{\nabla}(s \otimes u \otimes t) \in \Gamma\left( \CV^{\,i-1,0} \otimes \Omega_B^{k+1} \otimes \Omega_B^{j+1} \right).
\]
To prove the commutativity of the diagram, it suffices to show that $\overline{\nabla}(s \otimes t \otimes u)$ and $\overline{\nabla}(s \otimes u \otimes t)$ coincide under the natural isomorphism switching the second and the third factors
\[
\CV^{\,i-1,0} \otimes \Omega_B^{j+1} \otimes \Omega_B^{k+1} \simeq \CV^{\,i-1,0} \otimes \Omega_B^{k+1} \otimes \Omega_B^{j+1}. 
\]

By Griffiths transversality, the morphism $\overline{\nabla}$ of \eqref{4.1_0} is linear with respect to the second factor on the left-hand side as it represents $\CV^{0,j}$. Therefore, we have
\[
\overline{\nabla}(s \otimes t \otimes u) = \overline{\nabla}(s) \wedge t \wedge u.
\]
Here 
\begin{equation}\label{1.4_2}
\overline{\nabla}(s) \in \Gamma\left(\CV^{\,i-1,0} \otimes \Omega_B^1 \otimes \Omega_B^1\right)
\end{equation}
and the wedge product with $t$ and $u$ are on the second and  third factors, respectively. Hence the desired property is a consequence of Corollary~\ref{cor1.3}, which states that \eqref{1.4_2} is symmetric with respect to the second and  third factors.
\end{proof}

Lastly, we show that Theorem~\ref{thm0.3} follows from the commutative diagram of Proposition~\ref{prop1.4}. Since $\pi\colon  M\to B$ is smooth, the $\mathcal{D}_B$-module $\mathcal{P}_{i-n}$ in the Hodge module $P_{i-n}^H$ is the variation of Hodge structures~$\CV^{\,i}$, and the filtration $F_\bullet \CP_{i-n}$ is described as
\[
F_k \CP_{i-n} = F^{-k} \CV^{\,i}.
\]
In particular, the de Rham complex of $\mathcal{P}_{i-n}$ is
\[
\DR(\CP_{i-n}) = \left[\CV^{\,i} \xrightarrow{{\;\nabla\;}} \CV^{\,i} \otimes\Omega^1_B \xrightarrow{{\;\nabla\;}}  \cdots \xrightarrow{{\;\nabla\;}} \CV^{\,i} \otimes \Omega^n_B\right]\![n],
\]
and the associated perverse--Hodge complexes are
\[
\CG_{i,k} = \left[\CV^{\mkern1mu k,i-k} \xrightarrow{\;\overline{\nabla}\;} \CV^{\mkern1mu k-1,i-k+1} \otimes\Omega^1_B \xrightarrow{\;\overline{\nabla}\;}  \cdots \xrightarrow{\;\overline{\nabla}\;} \CV^{\mkern1mu k-n,i-k+n} \otimes \Omega^n_B\right]\![2n-i].
\]

We prove Theorem~\ref{thm0.3} by showing that the two complexes $\CG_{i,k}$ and $\CG_{k,i}$ match term by term via the isomorphisms \eqref{ijk}. For convenience we may assume  $i \leq k$. As $\CV^{\,i,j} = 0$ for $j<0$, we find
\begin{equation}\label{1.3_2}
\CG_{i,k} = \left[\CV^{\,i,0}\otimes \Omega_B^{k-i} \xrightarrow{\;\overline{\nabla}\;} \CV^{\,i-1,1}\otimes \Omega_B^{k-i+1}\xrightarrow{\;\overline{\nabla}\;} \cdots \xrightarrow{\;\overline{\nabla}\;} \CV^{\mkern1mu k-n,i-k+n} \otimes \Omega^n_B\right]\![2n-k].
\end{equation}
On the other hand, we have $\CV^{\,i,j} = 0$ for $j>n$ by the fact that $\pi$ is a family of abelian varieties of dimension~$n$. Consequently, the complex $\CG_{k,i}$ is of the form
\begin{equation}\label{1.3_3}
    \CG_{k,i} = \left[\CV^{\,i,k-i} \xrightarrow{\;\overline{\nabla}\;} \CV^{\,i-1,k-i+1}\otimes \Omega_B^{1}\xrightarrow{\;\overline{\nabla}\;} \cdots\xrightarrow{\;\overline{\nabla}\;} \CV^{\mkern1mu k-n,n} \otimes \Omega^{i-k+n}_B\right]\![2n-k].
\end{equation}
We see that the complexes \eqref{1.3_2} and \eqref{1.3_3} are of the same length and are both concentrated in degrees $[k-2n, i-n]$. Then Proposition~\ref{prop1.4} yields an isomorphism
\[
(\iota_{i,0,k-i}, \iota_{i-1,1,k-i+1}, \dots, \iota_{k-n,i-k+n,n})\colon  \CG_{i,k} \xrightarrow{\;\simeq\;} \CG_{k,i}, 
\]
where the commutative diagram guarantees that it is indeed an isomorphism of complexes. This completes the proof. \qed

\section{Hodge modules}\label{Sec2}

In order to extend the isomorphisms established in Section~\ref{Sec1} to the singular fibers, we need to use Saito's theory of Hodge modules \cite{S1, S2}. We begin with some relevant properties of Hodge modules. Then we recall Matsushita's result~\cite{Ma3} on the higher direct images of~$\mathcal{O}_M$; we show in Proposition~\ref{prop2.3} that Matsushita's result is equivalent to the case $k=2n$ of Conjecture~\ref{conj}.

\subsection{Hodge modules}
Recall that a variation of Hodge structures of weight $w$ on a nonsingular variety $X$ is a triple 
\begin{equation}\label{VHS0}
V^H = (\CV, F_\bullet, V), \quad \nabla(F_k \CV) \subset F_{k+1}\CV \otimes \Omega_X^1, 
\end{equation}
where $V$ is a ($\BQ$-)local system, $F_\bullet$ is an increasing filtration, and $\CV = V \otimes_\BQ \CO_X$, such that the restriction $(\CV_x, F_\bullet, V_x)$ to each $x\in X$ is a pure Hodge structure of weight $w$.\footnote{Traditionally the Hodge filtration is a decreasing filtration; the relation with the increasing filtration here is $F_{-k} = F^k$.} We say that~$V^H$ is polarizable if it admits a morphism 
\[
Q\colon  V \times V \lra \BQ(w) = (2\pi i)^{-w}\BQ
\]
inducing a polarization on each stalk $V_x$. 

Pure Hodge modules introduced by Saito \cite{S1} are vast generalizations of variations of Hodge structures. As in \eqref{VHS0}, a pure Hodge module on $X$ is a triple 
\begin{equation}\label{HM}
P^H = (\CP, F_\bullet, P), \quad \nabla(F_k\CP) \subset F_{k+1}\CP\otimes \Omega_X^1, 
\end{equation}
where $\CP$ is a regular holonomic $\CD_X$-module, $F_\bullet$ is a good filtration on $\CP$, and $P$ is a perverse sheaf on $X$, such that $\CP$ corresponds to $P\otimes_\BQ \BC$ via the Riemann--Hilbert correspondence.\footnote{For convenience, we sometimes only use the pair $(\CP, F_\bullet)$ to denote a pure Hodge module.} Such triples satisfy a number of technical conditions; in particular, one can define the notions of weight and polarization.

The following theorem by Saito \cite{S2} provides a concrete description of polarizable pure Hodge modules on $X$ as extensions of polarizable variations of Hodge structures.

\begin{thm}[Saito]\label{thm:Saito}\leavevmode
\begin{enumerate}
    \item\label{thm:Saito-1} The category $\HM(X,w)$ of polarizable pure Hodge modules of weight $w$ on $X$ is abelian and semisimple. 
    \item\label{thm:Saito-2} For any closed subvariety $Z \subset X$, a simple polarizable variation of Hodge structures of weight $w- \dim Z$ on a nonsingular open subset of $Z$ can be uniquely extended to a simple object in $\HM(X,w)$.
    \item All simple objects in $\HM(X,w)$ arise this way.
\end{enumerate}
\end{thm}

From now on, all Hodge modules are assumed to be pure and polarizable. For any simple Hodge module~\eqref{HM}, we define its support to be the support of the simple perverse sheaf $P$.

Set $\dim X = d$. The de Rham complex of a Hodge module $P^H = (\CP, F_\bullet)$ is 
\[
\DR(\CP) = \left[\CP \xrightarrow{\;\nabla\;} \CP\otimes\Omega^1_X \xrightarrow{\;\nabla\;}  \cdots \xrightarrow{\;\nabla\;} \CP \otimes \Omega^d_X\right]\![d]
\]
and is 
concentrated in degrees $[-d,0]$. The filtration $F_\bullet \CP$ induces an increasing filtration 
\[
F_k\DR(\CP) = \left[F_k\CP \xrightarrow{\;\nabla\;} F_{k+1}\CP\otimes\Omega^1_X \xrightarrow{\;\nabla\;}  \cdots \xrightarrow{\;\nabla\;} F_{k+d}\CP \otimes \Omega^d_X\right]\![d]
\]
whose $\supth{k}$ graded piece is the complex of $\CO_X$-modules
\[
\gr^F_k\DR(\CP) = \left[\gr^F_k\CP \xrightarrow{\;\overline{\nabla}\;} \gr^F_{k+1}\CP\otimes\Omega^1_X \xrightarrow{\;\overline{\nabla}\;}  \cdots \xrightarrow{\;\overline{\nabla}\;} \gr^F_{k+d}\CP \otimes \Omega^d_X\right]\![d].
\]
Note that $\gr^F_k\CP = 0$ for $k>0$ when $\CP$ is given by a variation of Hodge structures, but this is not true for general Hodge modules. The functor $\gr_k\DR(-)$ extends naturally to the bounded derived category of Hodge modules taking values in $D^b\Coh(X)$.

\subsection{Decomposition theorem, Saito's formula, and duality}
Let $f\colon  X\to Y$ be a~projective morphism between nonsingular varieties. For a Hodge module \mbox{$P^H = (\mathcal{P}, F_\bullet) \in \HM(X, w)$}, Saito's decomposition theorem \cite{S1} states that there is a decomposition in the bounded derived category of Hodge modules on $Y$,
\begin{equation} \label{saito_dec}
f_+P^H \simeq \bigoplus_i \mathcal{H}^i\left(f_+P^H\right)[-i],
\end{equation}
with $\mathcal{H}^i(f_+P^H) \in \HM(Y, w + i)$. Its compatibility with the functor $\gr^F_k\DR(-)$ is given by the following formula (often known as Saito's formula; see \cite[Section~2.3.7]{S1}):
\begin{equation} \label{saito_for}
Rf_*\gr^F_k\DR(\mathcal{P}) \simeq \gr^F_k\DR(f_+\CP) \simeq \bigoplus_i \gr^F_k\DR(\mathcal{H}^i(f_+\CP))[-i].
\end{equation}

The functor $\gr^F_k\DR(-)$ is also compatible with Serre duality. Recall that $\dim X = d$. For a Hodge module~$P^H = (\CP, F_\bullet) \in \HM(X, w)$, we have
\begin{equation}\label{duality}
R\mathcal{H}\mathrm{om}_{\CO_X}\left(\gr^F_k\DR(\CP), \omega_X[d]\right) \simeq \gr^F_{-k-w}\DR(\CP), 
\end{equation}
where $\omega_X$ is the dualizing sheaf of $X$; see \cite[Lemma 7.4]{Sch}.

Now we consider a Lagrangian fibration $\pi\colon  M \to B$ with $\dim M = 2n$. For our purpose, we study the direct image $\pi_+\BQ_M^H[2n]$ of the trivial Hodge module
\[
\BQ_M^H[2n] = (\CO_M, F_\bullet), \quad F_{-1}\CO_M =0 \subset F_0\CO_M  = \CO_M.
\]
A direct calculation yields
\begin{equation}\label{2.1_1}
\gr^F_{-k}\DR(\CO_M) \simeq \Omega_M^{k}[2n-k].
\end{equation}
Here conventionally $\Omega_M^k = 0$ for $k<0$. Formulas \eqref{saito_dec} and \eqref{saito_for} then read
\begin{equation*}
\pi_+\BQ_M^H[2n] \simeq \bigoplus_{i = -n}^n P^H_i[-i], \quad P^H_i = (\mathcal{P}_i, F_\bullet) \in \HM(B, 2n + i)
\end{equation*}
and
\begin{equation}\label{saito_formula}
R\pi_* \Omega_M^k[2n - k] \simeq \gr^F_{-k}\DR(f_+ \CO_M) \simeq \bigoplus_{i = -n}^n \gr^F_{-k}\DR(\CP_i)[-i].
\end{equation}
Finally, since $\dim B = n$ and all the fibers of $\pi\colon  M \to B$ have dimension $n$, we have $\gr_k^F\DR(\CP_i) = 0$ for~$k < \max\{-2n, -2n - i\}$. Applying the duality \eqref{duality}, we also have $\gr_k^F\DR(\CP_i) = 0$ for $k > \min\{0, -i\}$.

\subsection{Matsushita's theorem revisited}

Let $\pi\colon  M \to B$ be a Lagrangian fibration with $\dim M = 2n$. Matsushita \cite{Ma3} calculated the higher direct images of $\CO_M$.

\begin{thm}[Matsushita] For $0\leq i \leq n$, we have
\begin{equation}\label{Matsushita}
R^i\pi_* \CO_M \simeq \Omega_B^i.
\end{equation}
\end{thm}

\begin{proof}
Matsushita's proof assumes that $M$ is projective. However, as he explained in \cite[Remark~2.10]{Ma3}, the projectivity for $M$ is only used for Koll\'ar's decomposition \cite{K}
\[
R\pi_* \omega_M \simeq \bigoplus_i R^i\pi_*\omega_M[-i].
\]
Since  the decomposition now holds for any projective morphism as a consequence of Saito's theory of Hodge modules \cite{SK}, we may safely remove the projectivity assumption for $M$ in Matsushita's theorem.
\end{proof}

\begin{prop}\label{prop2.3}
The case $k=2n$ of Conjecture~\ref{conj} is equivalent to \eqref{Matsushita}. In particular, Conjecture~\ref{conj} holds when~$k = 2n$.
\end{prop}

\begin{proof}
When $k=2n$, the desired isomorphism of Conjecture~\ref{conj} is
\begin{equation}\label{2.3_1}
\gr^F_{-2n}\DR(\CP_{i-n})[n-i] \simeq \gr^F_{-i}\DR(\CP_{n})[-n].
\end{equation}
On one hand, the Hodge module $P_n^H$ is the $\supth{(-n)}$ Tate twist of $\BQ_B^H[n]$
\[
P^H_n \simeq \BQ_B^H[n](-n) = (\CO_B, F_{\bullet -n}),
\]
so the right-hand side of \eqref{2.3_1} is
\[
\gr^F_{-i}\DR(\CP_{n})[-n] \simeq  \Omega_B^{i-n}[n-i].
\]
On the other hand, since $\gr^F_{-2n}\DR(\CP_{i-n})$ is always a sheaf (concentrated in degree 0), combining \eqref{2.1_1} and~\eqref{saito_formula}, we obtain that the left-hand side of \eqref{2.3_1} is 
\begin{equation}\label{2.3_2}
\gr^F_{-2n}\DR(\CP_{i-n})[n-i] \simeq R^{i-n}\pi_*\omega_M [n-i] \simeq R^{i-n}\pi_*\CO_M [n-i].
\end{equation}
Consequently, \eqref{2.3_1} is equivalent to \eqref{Matsushita}.
\end{proof}

\begin{rmk}
In view of \eqref{saito_formula}, Conjecture~\ref{conj} would provide a new recipe for calculating the higher direct images of $\Omega^k_M$ for all $k$, extending Matsushita's work. We could use the perverse--Hodge symmetry to trade the contributions of $\gr_{-k}^F\DR(\CP_i)$ for all $i$ for the contributions of~$\gr_{-i}^F\DR(\CP_{k - n})$ for all $i$. The latter has the advantage of only involving a single Hodge module~$P_{k - n}^H$.
\end{rmk}

\subsection{An example}\label{Sect2.4}
To illustrate the subtleties when extending Theorem~\ref{thm0.3} to the singular fibers, we consider the following basic example.

Let $\pi\colon  S \to B$ be an elliptic fibration of a symplectic surface. We assume that $\pi$ only has singular fibers with a double point, over a finite set $D \subset B$. Let $j\colon  U = B \backslash D \hookrightarrow B$ be the open embedding.

We look at the symmetry
\begin{equation} \label{toy}
\gr^F_{-1}\DR(\CP_{-1})[1] \simeq \gr^F_0\DR(\CP_0)
\end{equation}
proposed by Conjecture~\ref{conj}. Since the Hodge module $P_{-1}^H$ is the trivial Hodge module $\BQ_B^H[1]$, the left-hand side of \eqref{toy} is $\Omega_B^1[1]$. For the right-hand side, let $(\CV, F_\bullet, V)$ denote the variation of Hodge structures~$R^1\pi_{U*}\BQ_{S_U}$ on $U$. Then the Hodge module $P_0^H$ on $B$ is the minimal extension~$(j_{!*}\CV, F_\bullet)$, which can be described concretely using Deligne's canonical extension.

Recall that the canonical extension depends on a real interval $[a, a+1)$ or $(a, a + 1]$ where the eigenvalues of the residue endomorphism should lie. In our situation, the monodromy around each point of $D$ is unipotent (given by the matrix $(\begin{smallmatrix} 1&1\\0&1 \end{smallmatrix})$ in local coordinates), so the eigenvalues are necessarily integers. Let~$\overline{\CV}$ be the canonical extension of $\CV$ with respect to either $[0, 1)$ or~$(-1, 0]$; it is locally free of rank $2$ on $B$. Schmid's theorem says that~$F_\bullet \overline{\CV} := j_*F_\bullet \CV \cap \overline{\CV}$ is a filtration by locally free subsheaves. By \cite[Section~3.10]{S2},
we have
\[
j_{!*}\CV = \CD_B \cdot \overline{\CV} \subset \overline{\CV}(*D), 
\]
where the $\CD_B$-action is induced by Deligne's meromorphic connection on $\overline{\CV}(*D)$, and
\[
F_kj_{!*}\CV = \sum_{i \geq 0} F_i\CD_B \cdot F_{k - i}\overline{\CV}.
\]

It follows that for the right-hand side of \eqref{toy}, we have
\begin{equation}\label{gr0}
\gr_0^F\DR(\CP_0) = \left[\frac{\overline{\CV} + F_1\CD_B \cdot F_{-1}\overline{\CV}}{F_{-1}\overline{\CV}} \xrightarrow{\;\overline{\nabla}\;} \frac{F_1\CD_B \cdot \overline{\CV} + F_2\CD_B \cdot F_{-1}\overline{\CV}}{\overline{\CV} + F_1\CD_B \cdot F_{-1}\overline{\CV}}\otimes \Omega_B^1\right]\![1].
\end{equation}
In particular, as a complex it has two nontrivial terms. But $\overline{\nabla}$ is clearly surjective; to see its kernel, we do a calculation in local coordinates. Let $t$ be the local coordinate of $B$ near $0 \in D$, and let $\alpha, \beta$ be a local trivialization of $\overline{\CV}$. Since the monodromy matrix around $0$ is $(\begin{smallmatrix} 1&1\\0&1 \end{smallmatrix})$, the residue matrix for $\overline{\CV}$ is $(\begin{smallmatrix} 0&1/2\pi\sqrt{-1}\\0&0 \end{smallmatrix})$. In other words, we have
\[
\nabla\alpha = 0, \quad \nabla\beta = \frac{1}{2\pi\sqrt{-1}}\alpha \otimes \frac{dt}{t}.
\]
We also have
\[F_{-1}\overline{\CV} = \langle f(t)\alpha + g(t)\beta \rangle,\]
where $f(t), g(t)$ are holomorphic functions with $g(t)$ nonvanishing. From this we see that
\[F_1\CD_B \cdot \overline{\CV} = \overline{\CV} + F_1\CD_B \cdot F_{-1}\overline{\CV},\]
hence
\begin{equation} \label{gr01}
\overline{\nabla}\left(\frac{\overline{\CV}}{F_{-1}\overline{\CV}}\right) = 0.
\end{equation}
On the other hand, the map $\overline{\nabla}$ induces an isomorphism
\begin{equation} \label{gr02}
\overline{\nabla}\colon  \frac{\overline{\CV} + F_1\CD_B \cdot F_{-1}\overline{\CV}}{\overline{\CV}} \xrightarrow{\;\simeq\;} \frac{F_1\CD_B \cdot \overline{\CV} + F_2\CD_B \cdot F_{-1}\overline{\CV}}{F_1\CD_B \cdot \overline{\CV}}\otimes \Omega_B^1
\end{equation}
sending $\frac{1}{t}\alpha$ to $-\frac{1}{t^2}\alpha \otimes dt$. Combining \eqref{gr01} and \eqref{gr02}, we deduce that the kernel of $\overline{\nabla}$ in \eqref{gr0} is
\[\ker(\overline{\nabla}) = \frac{\overline{\CV}}{F_{-1}\overline{\CV}} = \gr_0^F\overline{\CV}.\]

Finally, by \cite[Theorem 2.6]{K} and \eqref{Matsushita}, we have
\[\gr_0^F\overline{\CV} \simeq R^1\pi_*\CO_S \simeq \Omega_B^1, 
\]
which yields the desired isomorphism \eqref{toy} only(!) in the derived category $D^b\Coh(B)$. 

Note that the proof in Section~\ref{Sec3.1} works for \emph{all} symplectic surfaces $S$ with $\pi\colon  S \to B$ and does not rely on information about the singular fibers. 

\section{Hilbert schemes of points}
In this section we prove Theorem~\ref{thm0.4}; it is completed by a series of compatibility results regarding the perverse--Hodge symmetry and natural geometric operations.

\subsection{Surfaces}\label{Sec3.1}
We first verify Theorem~\ref{thm0.4} for $n=1$, where the Lagrangian fibration is an elliptic surface $\pi\colon  S \to B$. 
It suffices to prove Conjecture~\ref{conj} for
\[
0 \leq i < k \leq 2.
\]
The cases when $k = 2$ were covered by Proposition~\ref{prop2.3}. Therefore, it remains to show the symmetry \eqref{toy} for~$S$, whose left-hand side is
\[
\gr_{-1}^F \DR(\CP_{-1})[1] \simeq \Omega_B^1[1].
\]
The right-hand side can be computed via the duality \eqref{duality}:
\[
\gr^F_{0}\DR(\CP_0) \simeq
R\mathcal{H}\mathrm{om}_{\CO_B}\left(\gr^F_{-2}\DR(\CP_0), \omega_B[1]\right) \simeq \Omega_B^1[1],
\]
where the last isomorphism follows from \eqref{2.3_2}.

\subsection{Closed embeddings and finite morphisms}
Let $X$ be an irreducible quasi-projective variety of dimension $d$. Let \begin{equation}\label{PHS}
    \left\{Q^{H}_i = (\CQ_i, F_\bullet)\right\}_{i\in \BZ}
\end{equation}
be a finite sequence of Hodge modules on $X$. We say that the Hodge modules in \eqref{PHS} are \emph{perverse--Hodge symmetric} (PHS for short) on $X$ if for any $i,k\in \BZ$ we have
\[
\gr^F_{-k} \DR(\CQ_{i-d})[d - i] \simeq \gr^F_{-i} \DR(\CQ_{k-d})[d-k].\footnote{By \cite[Lemma 7.3]{Sch}, the functor $\gr_k\DR(-)$ is well defined for possibly singular varieties.}
\]
Clearly Conjecture~\ref{conj} is equivalent to the statement that the Hodge modules $P^H_i$ given by the decomposition theorem \eqref{DT} are PHS on $B$. In general we say that a morphism $f\colon  X \to Y$ is PHS if the trivial Hodge module $\BQ^H_X[\dim X]$ is pure on $X$ and the Hodge modules obtained from the decomposition theorem of~$f_+\BQ^H_X[\dim X]$ are PHS on $Y$.

The following proposition shows the compatibility between the perverse--Hodge symmetry and push-forwards along closed embeddings and finite morphisms.

\begin{prop}\label{prop3.1}
Assume that the Hodge modules $Q^{H}_i = (\CQ_i, F_\bullet)$ are PHS on $Z$.
\begin{enumerate}
    \item\label{prop3.1-1} Let $\iota\colon  Z \hookrightarrow X$ be a closed embedding of codimension $c$. Then the Hodge modules 
\[
Q_i'^H = \iota_{+} Q^{H}_i (-c) 
\]
are PHS on $X$.
    \item\label{prop3.1-2} If $f\colon  Z \to X$ is a finite surjective morphism with $\dim X = \dim Z$, then the Hodge modules 
    \[
Q_i'^H =   f_+Q^H_i
    \]
are PHS on $X$.
\end{enumerate}

\end{prop}

\begin{proof}
We only prove~\eqref{prop3.1-1} as~\eqref{prop3.1-2} is completely parallel. For a closed embedding $\iota\colon  Z \hookrightarrow X$, we have
\[
\begin{split}
     \gr^F_{-k} \DR(\CQ_{i-d}')[d-i]   & \simeq \iota_* \gr^F_{-k} \DR(\CQ_{i-d}(-c))[d-i]  \\
    & \simeq \iota_* \gr^F_{-{(k-c)}} \DR(\CQ_{(i-c)-\dim Z})[\dim Z-(i - c)].
\end{split}
\]
Similarly,  
\[
 \gr^F_{-i} \DR(\CQ_{k-d}')[d-k]   \simeq \iota_* \gr^F_{-{(i-c)}} \DR(\CQ_{(k-c)-\dim Z})[\dim Z-(k - c)].
\]
The proposition then follows from the isomorphism 
\[
\gr^F_{-{(k-c)}} \DR(\CQ_{(i-c)-\dim Z})[\dim Z - (i-c)] \simeq  \gr^F_{-{(i-c)}} \DR(\CQ_{(k-c)-\dim Z})[\dim Z - (k - c)]
\]
given by the assumption.
\end{proof}

\subsection{External products}
Let $X$ and $Y$ be quasi-projective varieties, and let 
\[
P^H = (\CP, F_\bullet), \quad  Q^H = (\CQ, F_\bullet)
\]
be Hodge modules on $X$ and $Y$, respectively. We recall the following standard lemma concerning the external product
\[
\CP\boxtimes \CQ = \pr_X^* \CP \otimes \pr_Y^* \CQ
\]
on $X\times Y$.

\begin{lem}\label{lem3.2}
We have
\[
\gr^F_k\DR(\CP\boxtimes \CQ) \simeq \bigoplus_{i+j=k} \gr^F_i\DR(\CP)\boxtimes \gr^F_j\DR(\CQ) \in D^b\Coh(X\times Y).
\]
\end{lem}

\begin{proof}
This follows from the fact that the (filtered) de Rham functor is compatible with taking external product (\textit{cf.}~\cite[Equation~(1.4.1)]{MSS}).
\end{proof}

Now we consider projective morphisms
\begin{equation}\label{f_i}
f_j\colon  X_j \lra Y_j, \quad  j = 1,2, \dots, n
\end{equation}
with $X_j$ nonsingular. For each $f_j$, we have the Hodge modules $Q^H_{i,j}$ obtained from the decomposition theorem
\[
f_+ \BQ^H_{X_j}[\dim X_j] \simeq \bigoplus_i Q^H_{i,j}[-i],\quad Q^H_{i,j} = \CH^i\left(f_+\BQ_{X_j}^H[\dim X_j]\right).
\]

We show the compatibility between the perverse--Hodge symmetry and products of varieties.

\begin{prop}\label{Prop3.3}
If the morphisms \eqref{f_i} are PHS, then the product morphism
\[
f = \Pi_j f_j\colon    X = X_1\times X_2 \times \cdots \times X_n \to  Y = Y_1\times Y_2 \times \cdots \times Y_n
\]
is also PHS.
\end{prop}

\begin{proof}
Since 
\[
\BQ_X^H [\dim X] \simeq \boxtimes_{j=1}^n \BQ^H_{X_i}[\dim X_i],
\]
we have
\[
f_+ \BQ_X^H [\dim X] \simeq \boxtimes_{j=1}^n f_{j+} \BQ^H_{X_j}[\dim X_j].
\]
Therefore, from
the decomposition theorem
\[
f_+  \BQ_X^H [\dim X] \simeq \bigoplus_i W^H_i [-i], \quad W^H_i =  \CH^i\left(f_+\BQ_X^H[\dim X]\right), 
\]
we obtain each summand
\[
W^H_i \simeq \bigoplus_{i_1 + \dots +i_n = i}  Q^H_{i_1,1} \boxtimes Q^H_{i_2,2} \boxtimes \cdots \boxtimes Q^H_{i_n,n}.
\]
By Lemma~\ref{lem3.2} this further yields
\[
\gr_{-k}^F\DR(\CW_{i-\dim Y}) \simeq \bigoplus_{\substack{i_1+ \cdots +i_n = i\\k_1 + \dots +k_n = k}} \gr_{-k_1}^F\DR(\CQ_{i_1-\dim Y_1, 1}) \boxtimes \cdots \boxtimes  \gr_{-k_n}^F\DR(\CQ_{i_n-\dim Y_n, n}). \]
Using this decomposition and the fact that the $Q^H_{i,j}$ are PHS, we see that there is a one-to-one correspondence between the summands in the decompositions of 
\[
\gr_{-k}^F\DR(\CW_{i-\dim Y})[\dim Y - i],  \quad \gr_{-i}^F\DR(\CW_{k-\dim Y})[\dim Y - k], 
\]
respectively. This completes the proof.
\end{proof}

\begin{rmk}
If each $f_j\colon  X_j \to Y_j$ is a Lagrangian fibration, then the product $f\colon  X \to Y$ is also a Lagrangian fibration. Hence Proposition~\ref{Prop3.3} provides consistency checks for Conjecture~\ref{conj}; it shows that if each $f_j$ satisfies Conjecture~\ref{conj}, then their product satisfies it as well.
\end{rmk}

\begin{rmk}
We note that the only use of the nonsingular assumption in Proposition~\ref{Prop3.3} is that the trivial Hodge modules $\BQ_{X_j}^H[\dim X_j]$ on $X_j$ are pure. 
\end{rmk}

\subsection{Symmetric products}

Let $Q^H = (\CQ, F_\bullet)$ be a Hodge module on $X$. Its symmetric product $(Q^H)^{(n)}$ was introduced in \cite{MSS}, which defines a Hodge module on the symmetric product $X^{(n)}$ of the variety. Furthermore, such an operation is extended to the bounded derived category of Hodge modules on $X$.

\begin{prop}\label{Prop3.5}
If the $Q^H_i$ are PHS on $X$, then the $(Q^H_i)^{(n)}$ are PHS on $X^{(n)}$.
\end{prop}

\begin{proof}
Let $q\colon  X^n \to X^{(n)}$ be the $\mathfrak{S}_n$-quotient map. For an object $\CF^\bullet \in D^b\Coh(X)$, we similarly consider the symmetric product
\[
(\CF^\bullet)^{(n)} = \left(q_*  {\CF^\bullet}^{\boxtimes n}\right)^{\mathfrak{S}_n} \in D^b\Coh(X^{(n)}).
\]
Now by Proposition~\ref{Prop3.3}, we know that the Hodge modules $(Q_i^H)^{\boxtimes n}$ are PHS on $X^n$. Proposition~\ref{prop3.1}\eqref{prop3.1-2} further implies that the Hodge modules $q_+  (Q_i^H)^{\boxtimes n} $ are PHS on $X^{(n)}$. To prove the corresponding property for 
\[
(Q^H_i)^{(n)} =  \left(q_+  (Q_i^H)^{\boxtimes n}\right)^\mathfrak{S_n}
\]
on $X^{(n)}$, it suffices to show that the natural isomorphism
\[
q_* \gr_k \DR\left( (Q_i^H)^{\boxtimes n} \right) \simeq \gr_k \DR\left( q_+  (Q_i^H)^{\boxtimes n} \right)
\]
is equivariant with respect to the $\mathfrak{S}_n$-actions. It follows from \cite{MSS} that the (filtered) de Rham functor is compatible with the symmetric group action on $X^n$; more precisely, see \cite[proof of Proposition 1.5]{MSS}.
\end{proof}

%taking symmetric product is compatible with the functor $\gr^F_k\DR(-)$:
%\[
%\gr^F_k\DR( (Q^H_i)^{(n)} ) \simeq \gr^F_k\DR( Q^H_i)^{(n)} \in D^b\Coh(X^{(n)}).
%\]
%This is established in \cite{MSS} that the (filtered) de Rham functor is compatible with the symmetric group action on $X^n$; more precisely see \cite[Proof of Proposition 1.5]{MSS}.
%\end{proof}

\subsection{Proof of Theorem~\ref{thm0.4}}
For our purpose, we describe the decomposition theorem associated with the morphism
\[
\pi^{[n]}\colon  S^{[n]} \xrightarrow{\;\,f\;} S^{(n)} \xrightarrow{\pi^{(n)}} C^{(n)}.
\]

The first map $f\colon  S^{[n]}\to S^{(n)}$ of the composition is semismall, and the associated decomposition theorem is calculated in \cite{Go2}, which we now review. For a partition 
\begin{equation}\label{partition}
\nu = 1^{a_1}2^{a_2}\cdots n^{a_n}
\end{equation}
of $n$, we use $S^{(\nu)}$ to denote the variety
\[
S^{(\nu)} = S^{(a_1)}\times S^{(a_2)} \times \cdots \times S^{(a_n)}.
\]
We consider the stratification of the target variety $S^{(n)}$ by the combinatorial types of the points:
\[
S^{(n)} = \bigsqcup_{\nu} S^{(n)}_\nu;
\]
there is a canonical finite surjective morphism
\[
\kappa_\nu\colon  S^{(\nu)} \lra \overline{S^{(n)}_\nu};
\]
see \cite[Section 3]{Go2}. We use $|\nu|$ to denote the length $\sum_ia_i$ of the partition \eqref{partition}. Then we have 
\[
\codim_{S^{(n)}}\left(\overline{S_\nu^{(n)}}\right) = 2(n-|\nu|).
\]
The main result of \cite{Go2} is the decomposition theorem
\begin{equation*}
    f_+ \BQ^H_{S^{[n]}}[2n] \simeq \bigoplus_{\nu} {\kappa_{\nu}}_+ \BQ^H_{S^{(\nu)}}[2|\nu|](|\nu|-n).
\end{equation*}

Composing with the symmetric product map $\pi^{(n)}\colon  S^{(n)} \to C^{(n)}$ induced by $\pi\colon  S \to C$, we have
\begin{equation}\label{GS}
\pi^{[n]}_+ \BQ^H_{S^{[n]}}[2n] \simeq \bigoplus_{\nu}  \pi^{(n)}_+{\kappa_{\nu}}_+ \BQ^H_{S^{(\nu)}}[2|\nu|](|\nu|-n).
\end{equation}

Now we consider the commutative diagram
\begin{equation*}
    \begin{tikzcd}[column sep=small]
    S^{(\nu)} \arrow[d, "\pi^{(\nu)}"] \arrow[rr, "\kappa_\nu"] & & \overline{S^{(n)}_\nu} \arrow[d, "\pi^{(n)}"] \\
      C^{(\nu)}  \arrow[rr, "\kappa'_\nu"] & & \overline{C^{(n)}_\nu}\rlap{,}
\end{tikzcd}
\end{equation*}
where $\kappa'_\nu\colon  C^{(\nu)} \to \overline{C^{(n)}_\nu}$ is defined analogously to $\kappa_\nu$, $\pi^{(\nu)}$ is induced by $\pi$, and the right vertical arrow is the restriction of $\pi^{(n)}\colon  S^{(n)} \to C^{(n)}$. Since
\[
\codim_{C^{(n)}}\left(\overline{C_\nu^{(n)}}\right) = \frac{1}{2} \codim_{S^{(n)}}\left(\overline{S_\nu^{(n)}}\right) = n-|\nu|,
\]
the right-hand side of \eqref{GS} can be expressed as 
\begin{equation}\label{decomp}
\bigoplus_{\nu}  {\kappa'_{\nu}}_+ \left( \pi^{(\nu)}_+ \BQ^H_{S^{(\nu)}}[\dim S^{(\nu)}] \left(-\codim_{C^{(n)}}\left(\overline{C_\nu^{(n)}}\right)\right) \right). 
\end{equation}

We complete the proof of Theorem~\ref{thm0.4} by showing that the Hodge modules given by each term in \eqref{decomp} are PHS on $C^{(n)}$. 

Since $\pi\colon  S \to C$ is PHS, by Proposition~\ref{Prop3.3}, $\pi^k\colon  S^k \to C^k $ is PHS for any $k \geq 1$. Combining with Proposition~\ref{Prop3.5}, we obtain that each symmetric product $\pi^{(k)}\colon  S^{(k)} \to C^{(k)}$ is PHS, which further yields that~$\pi^{(\nu)}\colon  S^{(\nu)} \to C^{(\nu)}$ is PHS by taking products. Equivalently, the Hodge modules $Q^H_{i,\nu}$ obtained from the decomposition theorem 
\begin{equation*}\label{decomp2}
\pi^{(\nu)}_+ \BQ^H_{S^{(\nu)}}\left[\dim S^{(\nu)}\right] \simeq \bigoplus_i Q^H_{i,\nu}[-i], \quad Q^H_{i,\nu}= \CH^i\left( \pi^{(\nu)}_+ \BQ^H_{S^{(\nu)}}\left[\dim S^{(\nu)}\right] \right)
\end{equation*}
are PHS on $C^{(\nu)}$. Finally, we push forward the Tate-twisted Hodge modules
\begin{equation}\label{HM0}
Q^H_{i, \nu} \left(-\codim_{C^{(n)}}\left(\overline{C_\nu^{(n)}}\right)\right)
\end{equation}
along the composition of the finite surjective map $C^{(\nu)} \to \overline{C^{(n)}_\nu}$ with the closed embedding~$\overline{C^{(n)}_\nu} \hookrightarrow C^{(n)}$. By Proposition~\ref{prop3.1}, the Hodge modules \eqref{HM0} are PHS on $C^{(n)}$ as desired, where the Tate twist in \eqref{HM0} is crucial. 
\qed

\section{Global cohomology and LLV algebras}
We specialize to Lagrangian fibrations $\pi\colon  M \to B$ associated with compact irreducible symplectic varieties. We prove Theorem~\ref{thm0.6}, which is the perverse--Hodge symmetry at the level of global cohomology. Our main tool is the Looijenga--Lunts--Verbitsky (LLV for short) Lie algebra \cite{Ver95, Ver96, LL}, a structure that is unique to compact irreducible symplectic varieties. 

\subsection{LLV (sub)algebras}
Let $M$ be a compact irreducible symplectic variety or, equivalently, a projective hyper-K\"ahler manifold. Assume  $\dim M = 2n$. We call an element~$\alpha \in H^2(M, \BC)$ of \emph{Lefschetz type} if for any~$k \geq 0$, cupping with $\alpha^k$ gives an isomorphism
\[
\alpha^k \cup\colon  H^{2n - k}(M, \BC) \xrightarrow{\;\simeq\;} H^{2n + k}(M, \BC).
\]
Such a class $\alpha$ induces an $\mathfrak{sl}_2$-triple $(L_\alpha, H, \Lambda_\alpha)$ acting on $H^*(M, \BC)$. The LLV algebra $\mathfrak{g}(M)$ is generated by all $\mathfrak{sl}_2$-triples associated with Lefschetz type classes. As is shown in \cite{Ver95, Ver96} and \cite{LL} independently, there is a natural isomorphism
\[\mathfrak{g}(M) \simeq \mathfrak{so}(b_2(M) + 2),\]
where $b_2(M)$ is the second Betti number of $M$.

Two subalgebras of $\mathfrak{g}(M)$ played a crucial role in establishing Theorem~\ref{thm1}. The first is Verbitsky's $\mathfrak{so}(5)$ generated by the $\mathfrak{sl}_2$-triples associated with the three K\"ahler classes~$\omega_I, \omega_J, \omega_K$. By \cite{Ver90}, the weight decomposition of $H^*(M, \BC)$ under Verbitsky's $\mathfrak{so}(5)$ coincides with the Hodge decomposition. Consider the Cartan subalgebra of this $\mathfrak{so}(5)$ spanned by
\[
H, H_F:=-\sqrt{-1}[L_{\omega_J}, \Lambda_{\omega_K}].
\]
Then the Hodge decomposition
\begin{equation} \label{hodged}
H^*(M, \BC) = \bigoplus_{i,j} H^{i, j}(M) = \bigoplus_{i, j}H^j(M, \Omega_M^i)
\end{equation}
satisfies
\[
H|_{H^{i,j}(M)} = (i + j - 2n) \id, \quad H_F|_{H^{i,j}(M)} = (i - j)\id.
\]

The second is the perverse $\mathfrak{so}(5)$ introduced in \cite{SY}, which concerns a Lagrangian fibration $\pi\colon  M \to B$. Let $\beta \in H^2(M, \BC)$ be the pull-back of an ample class on $B$, and let $\eta \in H^2(M, \BC)$ be a relatively ample class satisfying $q_M(\eta) = 0$; here $q_M(-)$ is the Beauville--Bogomolov--Fujiki form on $H^2(M, \BC)$. The perverse $\mathfrak{so}(5)$ is generated by the $\mathfrak{sl}_2$-triples associated with
\[
\eta + \beta,\; -\sqrt{-1}(\eta - \beta)
\]
and a third element $\rho \in H^2(M, \BC)$ satisfying
\[q_M(\rho) = q_M(\eta + \beta), \quad (\rho, \eta)_M = (\rho, \beta)_M = 0;\]
here $(-, -)_M$ is the bilinear form associated with $q_M(-)$.

By \cite[Theorem 3.1]{SY}, the weight decomposition of $H^*(M, \BC)$ under the perverse $\mathfrak{so}(5)$ takes the form
\begin{equation} \label{pervd}
H^*(M, \BC) = \bigoplus_{i, j} {^\mathfrak{p}H}^{i, j}(M) = \bigoplus_{i,j}H^{j - n}(B, P_{i - n}\otimes_\BQ \BC),
\end{equation}
where the $P_i$ are as in \eqref{decomp0}. In terms of the Cartan subalgebra spanned by
\[H, H_P:=-\sqrt{-1}\left[L_{\eta + \beta}, \Lambda_{-\sqrt{-1}(\eta - \beta)}\right],\]
we have
\[H|_{{^\mathfrak{p}H}^{i,j}(M)} = (i + j - 2n) \id, \quad H_P|_{{^\mathfrak{p}H}^{i,j}(M)} = (i - j)\id.\]

\subsection{Perverse--Hodge algebra}
Let $\pi\colon  M \to B$ be a Lagrangian fibration with $M$ a compact irreducible symplectic variety of dimension~$2n$. Since $B$ is of Picard rank $1$, the ample class~$\omega_{I} \in H^2(M, \BC)$ admits a unique decomposition $\omega_{I} = \eta + \beta$ with $\eta, \beta$ as in the previous section. We also have $\omega_J = \sigma + \overline{\sigma}$, where $\sigma$ is the symplectic form; hence $(\omega_J, \eta)_M = (\omega_J, \beta)_M = 0$. In particular, the perverse $\mathfrak{so}(5)$ can be generated by the $\mathfrak{sl}_2$-triples associated with $\omega_I = \eta + \beta, -\sqrt{-1}(\eta - \beta)$, and $\omega_J$.

We now consider the subalgebra $\mathfrak{g} \subset \mathfrak{g}(M)$ generated by the $\mathfrak{sl}_2$-triples associated with
\begin{equation} \label{four}
\omega_I = \eta + \beta,\; -\sqrt{-1}(\eta - \beta),\; \omega_J,\; \omega_K.
\end{equation}
We call it the \emph{perverse--Hodge algebra}; it is naturally isomorphic to $\mathfrak{so}(6)$ by the description of $\mathfrak{g}(M)$ in \cite[Theorem 11.1]{Ver95}. A Cartan subalgebra $\mathfrak{h} \subset \mathfrak{g}$ is spanned by
\begin{equation} \label{cartan}
H,\; H_P=-\sqrt{-1}\left[L_{\eta + \beta}, \Lambda_{-\sqrt{-1}(\eta - \beta)}\right],\; H_F=-\sqrt{-1}\left[L_{\omega_J}, \Lambda_{\omega_K}\right].
\end{equation}
We have the weight decomposition
\begin{equation} \label{pervhodged}
H^*(M, \BC) = \bigoplus_{i, k, d} H^{i, k, d}(M),
\end{equation}
so that
\[H|_{H^{i, k, d}(M)} = (d - 2n)\id, \quad H_P|_{H^{i, k ,d}(M)} = (2i - d)\id, \quad H_F|_{H^{i, k ,d}(M)} = (2k - d)\id.\]

The perverse--Hodge algebra $\mathfrak{g}$ contains both the perverse $\mathfrak{so}(5)$ and Verbitsky's $\mathfrak{so}(5)$ as subalgebras. Comparing \eqref{pervhodged} with \eqref{pervd} and \eqref{hodged}, we find
\[
H^{i, k, d}(M) = \gr_{-k}^FH^{d - 2n}(B, P_{i - n}\otimes_\BQ \BC[n - i]),
\]
where $F_{-k} = F^k$ is the Hodge filtration on the pure Hodge structure $H^{d - 2n}(B, P_{i - n}[n - i])$.

On the other hand, by Saito's formula \eqref{saito_for} applied to the Hodge module $P_{i - n}^H$ under the projection $f\colon  B \to \mathrm{pt}$, we have
\[
H^{d - 2n}\left(B, \gr_{-k}^F\DR(\CP_{i - n})[n - i]\right) \simeq \gr_{-k}^FH^{d - 2n}\left(B, P_{i - n}\otimes_\BQ \BC[n - i]\right).
\]
The left-hand side is precisely the cohomology of the perverse--Hodge complex $\CG_{i, k}$. We conclude that
\begin{equation} \label{weightid}
H^{i, k, d}(M) \simeq H^{d - 2n}(B, \CG_{i, k}).
\end{equation}

\subsection{Proof of Theorem~\ref{thm0.6}}
Via \eqref{weightid}, we have identified the cohomology groups $H^*(B, \CG_{i, k})$ with the weight spaces of $H^*(M, \BC)$ under the perverse--Hodge algebra $\mathfrak{g}$. Theorem~\ref{thm0.6} is then equivalent to the symmetry of the weight spaces
\begin{equation} \label{symm}
H^{i, k, d}(M) \simeq H^{k, i, d}(M).
\end{equation}

The symmetry is a feature of $\mathfrak{so}(6)$-representations. More concretely, consider the subspace
\[V = \langle\omega_I = \eta + \beta, -\sqrt{-1}(\eta - \beta), \omega_J, \omega_K \rangle \subset H^2(M, \BC)\]
equipped with the quadratic form $q_V = q_M|_V$. Set
\[\widetilde{V} = V \oplus \langle \mathbf{1} \rangle \oplus \langle\Omega\rangle, \quad q_{\widetilde{V}} = q_V \oplus \begin{pmatrix} 0 & -1 \\ -1 & 0 \end{pmatrix}.\]
By \cite[Theorem 11.1]{Ver95} we have a natural isomorphism  $\mathfrak{g} \simeq \mathfrak{so}(\widetilde{V})$. Up to renormalization, we may assume
\[(-\Omega, \mathbf{1})_{\widetilde{V}} = (\eta, \beta)_{\widetilde{V}} = (\sigma, \overline{\sigma})_{\widetilde{V}} = 1,\]
so that $-\Omega, \eta, \sigma, \mathbf{1}, \beta, \overline{\sigma}$ form a standard isotropic basis of $\widetilde{V}$. Under this basis, the three elements $H, H_P, H_F$ in \eqref{cartan} are precisely (twice) the basis of $\mathfrak{h} \subset \mathfrak{g}$ described in \cite[Section~18.1]{FH}.
Let~$H^*, H^*_P, H^*_F$ denote the dual basis of~$\mathfrak{h}^*$. By~\cite[Section~18.1, p.~271]{FH},
the Weyl group of $\mathfrak{g}$ is isomorphic to $(\BZ/2\BZ)^2 \rtimes \mathfrak{S}_3$, where the symmetric group $\mathfrak{S}_3$ permutes the three coordinate axes of $\mathfrak{h}^*$ and the generators of $(\BZ/2\BZ)^2$ act diagonally by $(1, -1, -1)$ and $(-1, -1, 1)$. In particular, there is an element of the Weyl group exchanging~$H^*_P, H^*_F$ while fixing $H^*$. Consequently, we obtain the symmetry of the weight spaces \eqref{symm} through the Weyl group action. \qed

\begin{rmk}
We have been kindly informed by Mirko Mauri that the isomorphism \eqref{symm} can be obtained by combining Verbitsky's $\mathfrak{so}(5)$-action and the monodromy symmetry; see the proof of \cite[Corollary 3.5]{HM2}.
\end{rmk}

\begin{rmk}
Starting from the isomorphism $H^{i + j - 2n}(B, \CG_{i, k}) \simeq H^{i + j - 2n}(B, \CG_{k, i})$, we may sum over the index~$k$. On one hand, we get
\[
\begin{split}
\bigoplus_{k = 0}^{2n} H^{i + j - 2n}(B, \CG_{i, k}) & \simeq \bigoplus_{k = 0}^{2n}\gr_{-k}^FH^{i + j - 2n}(B, P_{i - n}\otimes_\BQ \BC[n - i]) \\
& \simeq H^{i + j - 2n}(B, P_{i - n}\otimes_\BQ \BC[n - i]) \\
& \simeq H^{j - n}(B, P_{i - n}\otimes_\BQ \BC).
\end{split}
\]
On the other hand, by \eqref{saito_formula} we have
\[
\begin{split}
\bigoplus_{k = 0}^{2n} H^{i + j - 2n}(B, \CG_{k, i}) & \simeq H^{i + j - 2n}\left(B, \bigoplus_{k = 0}^{2n}\gr_{-i}^F\DR(\CP_{k - n})[n - k]\right) \\
& \simeq H^{i + j -2n}\left(B, R\pi_*\Omega_M^i[2n - i]\right) \\
& \simeq H^j(M, \Omega_M^i).
\end{split}
\]
We see that Theorem~\ref{thm0.6} refines Theorem~\ref{thm1}.
\end{rmk}

\begin{rmk}
One can collect the numbers $h^{i, k ,d} := \dim H^{i, k, d}(M)$ and depict them in a~$3$-space. We call it the \emph{perverse--Hodge diamond} of $\pi\colon  M \to B$. For example, the $(d=2n)$-plane has the shape
\[
\begin{matrix}
& & h^{n,2n,2n}\\
& \reflectbox{$\ddots$} & \vdots & \ddots\\
h^{0, n, 2n} & \cdots & h^{n, n, 2n} & \cdots & h^{2n, n, 2n}.\\
& \ddots & \vdots & \reflectbox{$\ddots$}\\
& & h^{n, 0, 2n}
\end{matrix}
\]
To simplify the discussion we assume $b_2(M) \geq 5$. Then the perverse--Hodge algebra $\mathfrak{g}$ can be upgraded to an $\mathfrak{so}(7)$ by adding one more $\mathfrak{sl}_2$-triple, namely the one associated with an element~$\rho \in H^2(M, \BC)$ which is orthogonal to the four classes in \eqref{four} with respect to $q_M(-)$ and shares the same norm. The Weyl group of $\mathfrak{so}(7)$ is the full symmetry group of the regular octahedron, and as such it acts on the perverse--Hodge diamond (whereas the Weyl group of $\mathfrak{so}(6)$ acts as the subgroup of rotational symmetries). Meanwhile, it is expected that the perverse--Hodge diamond has precisely the shape of a regular octahedron, meaning that no nonzero numbers $h^{i, k ,d}$ lie outside the convex hull of the six vertices
\[
h^{0, 0, 0},\; h^{0, n, 2n},\; h^{n, 0, 2n},\; h^{n, 2n, 2n},\; h^{2n, n, 2n},\; h^{2n, 2n, 4n}.
\]
See also \cite[Conjecture 1.19]{GKLR} for an even stronger conjecture. This expectation is verified for all known families of compact irreducible symplectic varieties in \cite[Theorem~1.23]{GKLR}. For even cohomology $H^{\even}(M, \BC)$, the expectation is shown in \cite[Theorem~5.2]{GKLR} and \cite[Corollary~3.4]{HM2} to be equivalent to Nagai's conjecture for type II degenerations of hyper-K\"ahler manifolds deformation equivalent to $M$. It is verified for $2n \leq 8$ in \cite[Proposition 6.5]{GKLR}; more generally, the $(d = 2n)$-plane and $(d = 2n - 2)$-plane cases are proven in \cite[Theorem 1.2]{HM2}. (Alternatively, one can show the two cases using octahedral symmetry and the knowledge that the border of the Hodge diamond of $M$ only has $1$'s and $0$'s.)
\end{rmk}

%%%%%%%%%%%%%%%%%%%%%
% References
%%%%%%%%%%%%%%%%%%%%%
\newcommand{\etalchar}[1]{$^{#1}$}

\end{document}